\documentclass[final,leqno,onefignum,onetabnum]{siamltex1213}
\usepackage{amsmath,amssymb,graphicx,mathdots}

\newcommand{\va}{{\bf a}}
\newcommand{\vb}{{\bf b}}

\newcommand{\dst}{\displaystyle}

\newcommand{\vu}{{\bf u}}
\newcommand{\vv}{{\bf v}}

\title{Companion Matrices and Their Relations to Toeplitz and Hankel Matrices A}

\author{Yousong Luo \and
Robin Hill \thanks{School of Mathematical and Geospatial Sciences, RMIT University, GPO Box 2476, Melbourne, Vic 3001, Australia (yousong.luo@rmit.edu.au, robin.hill@rmit.edu.au).}}

\begin{document}
\maketitle


\begin{abstract}
In this paper we describe some properties of companion matrices and demonstrate some special patterns that arise when a Toeplitz or a Hankel matrix is multiplied by a related companion matrix.  We present a new condition, generalizing known results,  for a Toeplitz or a Hankel matrix to be the transforming matrix for a similarity between a pair of companion matrices.
A special case of our main result shows that a Toeplitz or a Hankel matrix can be extended using associated companion matrices, preserving the Toeplitz or Hankel structure respectively.
\end{abstract}



%
%

\section{Introduction and notation}

Companion matrices occur in many scientific fields. In particular, a companion
matrix naturally arises as the system matrix when a dynamic system is
represented in state space form \cite{Kailath:1, Luen:1}. When
a basis of the state vector space is changed a new system matrix appears and,
very often, the new system matrix is also a companion matrix and the
similarity relation between the new and the old system matrices is realized by
a nonsingular Toeplitz (or Hankel) matrix. In the literature on dynamic systems
such similarity transformations are verified case by case. In this paper we
give a general condition for a Toeplitz or Hankel matrix, satisfaction of
which ensures that the Toeplitz (Hankel) matrix transforms one companion
matrix to another. In the second part of this paper we also investigate some
extensions of a matrix by companion matrices, and finally through an example
we indicate some applications. In a dynamic system, if the system matrix under
a basis of an $n$-dimensional state space is a companion matrix, then a
special case of the extension we introduce yields the well known
representation of the state vector at any future time instant in terms of a
given initial state vector together with knowledge of the input up to the
current time. We discuss applications in a discrete-time setting.
Continuous-time analogues are easily deduced, and are virtually identical.

Given vectors $\mathbf{u}=(u_{1},\ldots,u_{n+1})^{T}$ and $\mathbf{v}%
=(v_{1},\ldots,v_{n+1})^{T}\in\mathbb{R}^{n+1}$ we define the polynomials
\[
u(\lambda)=u_{1}+u_{2}\lambda+\cdots+u_{n}\lambda^{n-1}+u_{n+1}%
\lambda^{n}%
\]
and
\[
v(\lambda)=v_{1}+v_{2}\lambda+\cdots+v_{n}\lambda^{n-1}+v_{n+1}%
\lambda^{n}.
\]
We assume always that $u_{1}$, $u_{n+1}$, $v_{1}$ and $v_{n+1}$ are nonzero,
and that $u(\lambda)$ and $v(\lambda)$ are co-prime. The
\textquotedblleft top", \textquotedblleft bottom", \textquotedblleft left" and
\textquotedblleft right" companion matrices of the polynomial $u(\lambda)$ (or the vector $\mathbf{u}$) are
defined as
\[
C_{t}(\mathbf{u}):=\left[
\begin{array}
[c]{cc}%
\begin{array}
[c]{cc}%
\displaystyle-\frac{u_{n}}{u_{n+1}} & \cdots
\end{array}
& \displaystyle-\frac{u_{1}}{u_{n+1}}\\
I_{n-1} & 0
\end{array}
\right]  ,\quad C_{b}(\mathbf{u}):=\left[
\begin{array}
[c]{cc}%
0 & I_{n-1}\\
\displaystyle-\frac{u_{n+1}}{u_{1}} &
\begin{array}
[c]{cc}%
\cdots & \displaystyle-\frac{u_{2}}{u_{1}}%
\end{array}
\end{array}
\right]  ,
\]%
\[
C_{l}(\mathbf{u}):=\left[
\begin{array}
[c]{cc}%
\begin{array}
[c]{c}%
\displaystyle-\frac{u_{2}}{u_{1}}\\
\vdots
\end{array}
& I_{n-1}\\
\displaystyle-\frac{u_{n+1}}{u_{1}} & 0
\end{array}
\right]  \quad\mbox{and}\quad C_{r}(\mathbf{u}):=\left[
\begin{array}
[c]{cc}%
\displaystyle0 & \displaystyle-\frac{u_{1}}{u_{n+1}}\\
I_{n-1} &
\begin{array}
[c]{c}%
\vdots\\
\displaystyle-\frac{u_{n}}{u_{n+1}}%
\end{array}
\end{array}
\right]  .
\]
When their dependence on $\mathbf{u}$ is clear from context we will simply
write $C_{t}$, $C_{b}$, $C_{l}$ and $C_{r}$. The companion matrices of
$v(\lambda)$ are defined in the same way. Under our assumptions on $\mathbf{u}$ and $\mathbf{v}%
$, all the companion matrices defined above are nonsingular.

Let $J$ be the flipping matrix
\[
J=\left[
\begin{array}
[c]{ccc}%
0 &  & 1\\
& \iddots & \\
1 &  & 0
\end{array}
\right]  .
\]
For a vector $\mathbf{u}$ we denote by $\mathbf{u}^{J}$ the vector
$J\mathbf{u}$, and corresponding polynomial $u^{J}(\lambda)$ is defined by $u^{J}(\lambda)=u_{n+1}+u_{n}\lambda+\cdots+u_{2}\lambda^{n-1}+u_{1}%
\lambda^{n}$. For a matrix $A$ we denote by $A^{J}$ the flipping of $A$ about
its secondary diagonal, so $A^{J}=JA^{T}J$. Hankel matrices are symmetric in
the usual sense but Toeplitz matrices $A$ are persymmetric, that is, symmetric about their secondary
diagonal
\begin{equation}
A^{J}=A.\label{toep}%
\end{equation}
We also define the companion matrices of $\mathbf{u}^{J}$ and denote them by
$C_{t}(\mathbf{u}^{J})$, $C_{b}(\mathbf{u}^{J})$, $C_{l}(\mathbf{u}^{J})$ and
$C_{r}(\mathbf{u}^{J})$. When their dependence on $\mathbf{u}^{J}$ is clear
from context we will simply write these matrices as $\overline{C}_{t}$,
$\overline{C}_{b}$, $\overline{C}_{l}$ and $\overline{C}_{r}$.

Define the following triangular Toeplitz matrices using the components of
$\mathbf{u}$ and $\mathbf{v:}$
\[
U_{+}:=\left[
\begin{array}
[c]{cccc}%
u_{1} & 0 & \cdots & 0\\
u_{2} & u_{1} & \ddots & \vdots\\
\vdots & \ddots & \ddots & 0\\
u_{n} & \cdots & u_{2} & u_{1}%
\end{array}
\right]  \qquad U_{-}:=\left[
\begin{array}
[c]{cccc}%
u_{n+1} & u_{n} & \cdots & u_{2}\\
0 & u_{n+1} & \ddots & \vdots\\
\vdots & \ddots & \ddots & u_{n}\\
0 & \cdots & 0 & u_{n+1}%
\end{array}
\right].
\]%
Analogously we define $V_{+}$ and $V_{-}$ in terms of the components of  $\mathbf{v}$.

The Toeplitz Bezoutian $B_{T}:= \mathbf{Bez}%
_{T}(\mathbf{u},\mathbf{v})=\left(  b_{ij}\right)  _{i,j=1}^{n}$ and Hankel
Bezoutian $ B_{H}:= \mathbf{Bez}_{H}(\mathbf{u},\mathbf{v})=\left(  b_{ij}\right)
_{i,j=1}^{n}$ of the vectors $\mathbf{u},\mathbf{v}$ (or the polynomials
$u(\lambda),v(\lambda)$) are the $n\times n$ matrices with
the generating polynomials%

\begin{equation}
\sum_{i,j=1}^{n}b_{ij}\lambda^{i-1}\mu^{j-1}=\frac{u\left(
\lambda\right)  v^{J}\left(  \mu\right)  -u^{J}\left(
\mu\right)  v\left(  \lambda\right)  }{1-\mu\lambda} \label{def:BTdn}%
\end{equation}
and
\begin{equation}
\sum_{i,j=1}^{n}b_{ij}\lambda^{i-1}\mu^{j-1}=\frac{u\left(
\lambda\right)  v\left(  \mu\right)  -u\left(  \mu\right)
v\left(  \lambda\right)  }{\lambda-\mu} \label{def:BHdn}%
\end{equation}
respectively.  The Gohberg-Semencul formulae \cite{GohSem72,Bezou:1} imply that
the Toeplitz Bezoutian matrix generated by $\mathbf{u}$ and $\mathbf{v}$ is
\begin{equation}
B_{T}=U_{+}V_{-}-V_{+}U_{-}%
=V_{-}U_{+}-U_{-}V_{+}, \label{tbz}%
\end{equation}
and the Hankel Bezoutian matrix generated by $\mathbf{u}$ and $\mathbf{v}$ is
\begin{equation}
B_{H}=V_{+}JU_{-}-U_{+}JV_{-}%
=U_{-}JV_{+}-V_{-}JU_{+}. \label{hbz}%
\end{equation}
It is known \cite{lander} that if $u(\lambda)$ and $v(\lambda)$ are co-prime then
$B_{T}$ and $B_{H}$ are both nonsingular and that $B_{T}^{-1}$ is Toeplitz and
$B_{H}^{-1}$ is Hankel.


\section{Properties of companion matrices}

Here we list some obvious relations among the companion matrices defined in
Section 1.

\textbf{Properties:}

\begin{enumerate}
\item Inversion:
\begin{equation}
\label{cinverse}C_{t}=C_{b}^{-1}, \quad C_{l}=C_{r}^{-1}, \quad\overline
{C}_{t}=\overline{C}_{b}^{-1}, \quad\overline{C}_{l}=\overline{C}_{r}^{-1}.
\end{equation}

\item Flipping:
\begin{equation}
\label{flip}{C_{t}}^{J} = C_{r}, \quad{C_{b}}^{J} = C_{l}, \quad{\overline
{C}_{t}}^{J} = \overline{C}_{r}, \quad{\overline{C}_{b}}^{J} = \overline
{C}_{l}.
\end{equation}

\item Transposition:
\begin{equation}
\label{trans}{C_{t}}^{T} = \overline{C}_{l}, \quad{C_{b}}^{T} = \overline
{C}_{r}, \quad{C_{r}}^{T} = \overline{C}_{b}, \quad{C_{l}}^{T} = \overline
{C}_{t}.
\end{equation}

\end{enumerate}

Property 1 can be found in \cite{Brand:1}. All others can be easily verified.

\subsection{Similarity}

There are many similarities among the companion matrices. We are interested
here in finding matrices which realize such similarities. An obvious one (see
\cite{Kailath:1}), easily verifiable by a simple calculation, is $U_{+},$
because%
\[
U_{+}C_{t}U_{+}^{-1}=C_{r}\quad\mbox{and}\quad U_{+}C_{b}U_{+}^{-1}=C_{l}.
\]
Since $C_{b}=C_{t}^{-1}$ and $C_{l}=C_{r}^{-1}$, the second equation can be
written as $U_{+}C_{t}^{-1}U_{+}^{-1}=C_{r}^{-1}$. If, for any positive
integer $k$ we write $C_{t}^{-k}=\left(  C_{t}^{-1}\right)  ^{k}=C_{b}^{k}$
and $C_{r}^{-k}=\left(  C_{r}^{-1}\right)  ^{k}=C_{l}^{k}$, then the
similarity relation above extends to:
\begin{equation}
U_{+}C_{t}^{k}U_{+}^{-1}=C_{r}^{k}\label{similar1}%
\end{equation}
for all integers $k$.

Another nontrivial similarity relation, this time between ${C_{t}}^{T}$ and
$C_{t},$ is also given in \cite{Kailath:1}, where a linear dynamic system in continuous time is represented in the state space form
	\[ \dot{\alpha} (t) = A \alpha + w B, \quad y = D \alpha, \]
where  $\alpha (t) $ is the state vector, $A$ is the system matrix, $B$ is the input column matrix, $D$ is the output row matrix, $w$ is a scalar input and $y$ is the scalar output.  The state space representation is said to be in  canonical observer form if the system matrix $A_1$, the input matrix $B_{1}$ and the output matrix $D_{1}$ are
\[  A_1=\left[
\begin{array}
[c]{cc}%
\begin{array}
[c]{cc}%
\displaystyle-a_1 & \cdots
\end{array}
& \displaystyle-a_n\\
I_{n-1} & 0
\end{array}
\right] , \quad B_{1}=\left[
\begin{array}
[c]{c}%
1\\ 0 \\ \vdots \\
 0
\end{array}
\right], \quad D_{1} =\left[
\begin{array}{ccc} b_1 & \ldots & b_n \end{array}
\right].\]
In our notation $A_1 = C_{t}^T(\vu)$ with $\vu =( a_n,\ldots, a_1, 1)^T$.  The system representation is in canonical controller form if $A_2=C_{t}(\vu)$, $B_{2}=(b_1,\ldots, b_n)^{T}$ and the output matrix is $D_{1} =(1,0, \ldots, 0)$.
In the canonical observer form the controllability matrix $\mathcal{C}(A_1,B_{1})$ and the observability matrix $\mathcal{O}(D_{1},A_1)$ are then constructed as follows:
	\[ \mathcal{C}(A_1,B_{1}) =[ \begin{array}{cccc}
	B_{1} & A_1 B_{1} & \cdots & A_1^{n-1}B_{1}
	\end{array} ]\]
and
	\[\mathcal{O}(D_{1},A_1) = \left[ \begin{array}{c}
	 D_{1} \\ D_{1}A_1 \\ \vdots \\ D_{1}A_1^{n-1}
	\end{array} \right].\]
In the canonical controller form  $\mathcal{C}(A_2,B_{2})$ and $\mathcal{O}(D_{2},A_2)$ can be constructed in the same way. Under the condition that the system is both controllable and observable it can be shown that the matrix
\[
Q=\mathcal{O}^{-1}(D_{2},A_2)\mathcal{O}(D_{1},A_1)=\mathcal{C}%
(A_2,B_{2})\mathcal{C}^{-1}(A_1,B_{1})
\]
will transform $A_2$ into $A_1$ by way of $Q^{-1}A_2Q=A_1$, that is
\[
Q^{-1} C_{t}Q={C_{t}}^{T}=\overline{C}_{l}.
\]
It is shown in \cite{Kailath:1} that $Q=-{B_{T}}^{T}J$ where $B_{T}$ is the Toeplitz Bezoutian $\mathbf{Bez}%
_{T}(\mathbf{u},\mathbf{v})$ and hence $Q$ is a Hankel Bezoutian.  To make our notation consistent with the notation in \cite{Kailath:1} we have put $\vu =( a_n,\ldots, a_1, 1)^T$ and
$\vv=(b_n,\ldots, b_1, 0)^T$.   In general, for
all integers $k$
\begin{equation}
QC_{t}^{k}Q^{-1}=\overline{C}_{l}^{k}.\label{similar2}%
\end{equation}

In this section we introduce a general condition for a nonsingular matrix
to be a transforming matrix realizing a similarity between companion matrices.
We will show that both (\ref{similar1}) and (\ref{similar2}) are special cases of our general result.

To describe our generalization, we define a simple operation on square
Toeplitz or Hankel matrices. For an invertible Toeplitz matrix
\[
T=\left[
\begin{array}
[c]{cccc}%
a_{0} & a_{-1} & \cdots & a_{1-n}\\
a_{1} & a_{0} & \ddots & \vdots\\
\vdots & \ddots & \ddots & a_{-1}\\
a_{n-1} & \cdots & a_{1} & a_{0}%
\end{array}
\right]  ,
\]
the $(n-1)\times(n+1)$ Toeplitz matrix $\partial T$, introduced by \cite{HeinigRost}, is obtained by adding one
column to the right preserving the Toeplitz structure and then deleting the
first row:
\begin{equation}
\partial T :=\left[
\begin{array}
[c]{ccccc}%
a_{1} & a_{0} & a_{-1} & \cdots & a_{1-n}\\
\vdots & \ddots & \ddots & \ddots & \vdots\\
a_{n-1} & \cdots & a_{1} & a_{0} & a_{-1}%
\end{array}
\right]  .
\end{equation}
Similarly, for an invertible Hankel matrix $ H=TJ$, the $(n-1)\times(n+1)$ Hankel
matrix $\partial H$ is obtained by adding one column to the right preserving
the Hankel structure and then deleting the last row:
\begin{equation}
H=\left[
\begin{array}
[c]{cccc}%
a_{1-n} & \cdots & a_{-1} & a_{0}\\
\vdots & \iddots & a_{0} & a_{1}\\
a_{-1} & \iddots & \iddots & \vdots\\
a_{0} & a_{1} & \cdots & a_{n-1}%
\end{array}
\right]  ,\;\;\partial H :=\left[
\begin{array}
[c]{ccccc}%
a_{1-n} & \cdots & a_{-1} & a_{0} & a_{1}\\
\vdots & \iddots & \iddots & \iddots & \vdots\\
a_{-1} & a_{0} & a_{1} & \cdots & a_{n-1}%
\end{array}
\right]  .
\end{equation}

\begin{theorem}
\label{similarbyt} Suppose $T$ is an invertible $n\times n$ Toeplitz matrix,
and $\mathbf{u}=(u_{1},\ldots,u_{n+1})^{T}$ is a vector such that $u_{1}%
,u_{n+1}\neq0$. Then the following three statements are equivalent.

\begin{enumerate}
\item $\mathbf{u}\in\mbox{Ker}\{\partial T\}$.

\item  Both $C_{t}T$ and $TC_{r}$ are Toeplitz matrices and satisfy
$ C_{t}T=TC_{r}$.

\item  Both $C_{b}T$ and $TC_{l}$ are Toeplitz matrices and satisfy
$ C_{b}T=TC_{l}$.
\end{enumerate}

Furthermore, for all integers $k$,
\begin{equation}
T^{-1}C_{t}^{k}T=C_{r}^{k}. \label{similargent}%
\end{equation}

\end{theorem}

\begin{proof}
We prove that item 1 implies item 2 first.
We use the notation $T_{[i:j,k:l]}$ to denote the sub-matrix of $T$ formed by selecting all rows from the $i$th row to the $j$th row and all columns from the $k$th column to the $l$th column.
It is easy to see that
	\[ T C_r = \left[
	 \begin{array}{cc}
	    T_{[1:n,2:n]} &\beta
	 \end{array}
	 \right]
\]
where $\beta$ is a column given by
	\[ \beta = -\frac{1}{u_{n+1}}T  \left[ \begin{array}{c}
		 u_1  \\
		  \vdots \\
		   u_{n}
		 \end{array}
		 \right] = -\frac{1}{u_{n+1}} \left[ \begin{array}{c}
			 		 	T_{[1:1,1:n]}  \\
			 		 	T_{[2:n,1:n]}
			 		 \end{array}
			 		 \right]
		\left[ \begin{array}{c}
		 u_1  \\
		  \vdots \\
		   u_{n}
		 \end{array}
		 \right] . \]
Since  $ \vu \in \mbox{Ker} \{ \partial T \}$ we have
	\[ T_{[2:n,1:n]} \left[ \begin{array}{c}
			 			u_1  \\
			 			  \vdots \\
			 			  u_{n}
			 			 \end{array}
			 			 \right] +u_{n+1} \left[ \begin{array}{c}
			a_{1-n}  \\
			  \vdots \\
			  a_{-1}
			 \end{array}
			 \right]= 0. \]
This implies that
	\[ \beta
		 = \left[ \begin{array}{c}
			 		 	\mu_{-1}  \\
			 		 	a_{1-n} \\
			 		 			  \vdots \\
			 		 			  a_{-1}
			 		 \end{array}
			 		 \right] \]
where $\mu_{-1} = -\frac{1}{u_{n+1}} \left[ \begin{array}{cccccc}
a_0&  a_{-1} & \cdots  & a_{1-n} & 0
\end{array} \right]
\vu$, that is, $\beta^T J$ is the first row of $T C_r$.  From this we conclude that $T C_r$ is Toeplitz and hence $T C_r = (T C_r)^J$. On the other hand, since $T$ is Toeplitz, by equation (\ref{flip}) we have
	\[ \left( T C_r \right)^J = J\left(T C_r \right)^T J = (J\left( C_r  \right)^T J) (JT^T J) =  \left( C_r  \right)^J T^J
	=C_t T.\]
As a consequence we have
	\[ T C_r = C_t T.\]

Next we prove that item 2 implies item 1.  By equating the last columns of $T C_r$ and $ C_t T$ we see that
	\[ -\frac{1}{u_{n+1}}T  \left[ \begin{array}{c}
			 u_1  \\
			  \vdots \\
			   u_{n}
			 \end{array}
			 \right] = \left[ \begin{array}{c}
				 		 	\mu_1  \\
				 		 	a_{1-n} \\
				 		 			  \vdots \\
				 		 			  a_{-1}
				 		 \end{array}
				 		 \right] . \]
Deleting the first row on both sides we obtain $(\partial T) \vu =0$.

Finally we prove that item 3 is equivalent to item 1.
Similar to the argument above we have
	\[ T C_l = \left[
	 \begin{array}{cc}
	   \gamma & T_{[1:n,1:n-1]}
	 \end{array}
	 \right]
\]
where $\gamma$ is a column given by
	\[ \gamma = -\frac{1}{u_{1}}T  \left[ \begin{array}{c}
		 u_2  \\
		  \vdots \\
		   u_{n+1}
		 \end{array}
		 \right] = -\frac{1}{u_{1}} \left[ \begin{array}{c}
			 		 	T_{[1:n-1,1:n]}  \\
			 		 	T_{[n-1:n,1:n]}
			 		 \end{array}
			 		 \right]
		\left[ \begin{array}{c}
		 u_2  \\
		  \vdots \\
		   u_{n+1}
		 \end{array}
		 \right] . \]
Since  $ \vu \in \mbox{Ker} \{ \partial T \}$ we have
	\[ u_{1} \left[ \begin{array}{c}
			a_{1}  \\
			  \vdots \\
			  a_{n-1}
			 \end{array}
			 \right] + T_{[1:n-1,1:n]} \left[ \begin{array}{c}
			 			u_2  \\
			 			  \vdots \\
			 			  u_{n+1}
			 			 \end{array}
			 			 \right] = 0. \]
This implies that
	\[ \gamma
		 = \left[ \begin{array}{c}
			 		 	
			 		 	a_{1} \\
			 		 			  \vdots \\
			 		 			  a_{n-1} \\
                    \mu_1  \\
			 		 \end{array}
			 		 \right] \]
where $\mu_1 = -\frac{1}{u_{1}} \left[ \begin{array}{cccccc}
0 & a_n&  a_{n-1} & \cdots  & a_{1}
\end{array} \right]
\vu$, that is, $\gamma^T J$ is the last row of $T C_l$.  From this we conclude that $T C_l$ is Toeplitz and hence $T C_l = (T C_l)^J$. On the other hand, since $T$ is Toeplitz, by equation (\ref{flip}) we have
	\[ \left( T C_l \right)^J = J\left(T C_l \right)^T J = (J\left( C_l  \right)^T J) (JT^T J) =  \left( C_l  \right)^J T^J
	=C_b T.\]
As a consequence we have
	\[ T C_l = C_b T.\]
Thus item 1 implies item 3.  To see that item 3 implies item 1 we equate the first columns of $T C_l$ and $ C_b T$ to get
	\[ -\frac{1}{u_{1}}T  \left[ \begin{array}{c}
			 u_2  \\
			  \vdots \\
			   u_{n+1}
			 \end{array}
			 \right] = \left[ \begin{array}{c}
				 	a_{1} \\
				 \vdots \\
			  a_{n-1} \\
                \mu_1
				 		 \end{array}
				 		 \right] . \]
Deleting the last row on both sides we obtain $(\partial T) \vu =0$.

To see (\ref{similargent}) we write $T C_r = C_t T$ in the form $T^{-1} C_t T = C_r$ and then take integer powers on both sides.

\end{proof}

\begin{corollary}
\label{similarbyh} Suppose  $H$ is an invertible $n\times n$ Hankel matrix and
$\mathbf{u}=(u_{1},\ldots,u_{n+1})^{T}$ is a vector such that $u_{1}%
,u_{n+1}\neq0$. Then the following statements are equivalent.

\begin{enumerate}
\item  $\mathbf{u}^{J}%
\in\mbox{Ker}\{\partial H\}$.

\item Both $C_{t}H$ and $H\overline{C}_{l}$ are Hankel matrices and satisfy
   $C_{t}H=H\overline{C}_{l}$.

\item Both $C_{b}H$ and $H\overline{C}_{r}$ are Hankel matrices and satisfy
    $ C_{b}H=H\overline{C}_{r}$.
\end{enumerate}
Furthermore, for all integers $k$,
\begin{equation}
H^{-1}C_{b}^{k}H=\overline{C}_{l}^{k}. \label{similargenh}%
\end{equation}

\end{corollary}

\begin{proof} Let $T=HJ$. Then $T$ is a non-singular Toeplitz matrix.  Obviously $\mathbf{u}^{J} \in\mbox{Ker}\{\partial H\}$ is equivalent to $\mathbf{u}\in\mbox{Ker}\{\partial T\}$. Now we prove that item 2 of Theorem \ref{similarbyt} and item 2 of Corollary \ref{similarbyh} are equivalent, that is,  $T C_r = C_t T$ is equivalent to $C_t H = H \overline{C}_l$. Since $H$ is Hankel we then have
	\[ T^J=JT^TJ=(TJ)^TJ=H^TJ=HJ.\]
By taking the flipping  operation on both sides of $T C_r = C_t T$ we obtain $	C_r^JT^J = T^JC_t^J$ which is, by (\ref{flip}),
	$  C_tH = HC_t^T $.
Applying (\ref{trans}) we have
	\[ C_tH = H\overline{C}_l. \]
The proof for the equivalence of item 3 of Theorem \ref{similarbyt} and  Corollary \ref{similarbyh} follows in a similar way.  This completes the proof.
\end{proof}

\begin{corollary}
\label{shift} If $A$ is an invertible matrix such that $A^{-1}C_{t}A=C_{r}$,
then $B=C_{t}A$ (or $B=AC_{r}$) plays the same role as $A$, that is,
$B^{-1}C_{t}B=C_{r}$.

If $A$ is an invertible matrix such that $A^{-1}C_{t}A=\overline{C}_{l}$, then
$B=C_{t}A$ (or $B=A\overline{C}_{r}$) plays the same role as $A$, that is,
$B^{-1}C_{t}B=\overline{C}_{l}$.
\end{corollary}

The proof of Theorem \ref{similarbyt} is constructive. Given a pair of
companion matrices as stated in the Theorem, we can use the procedure in the
proof to find a Toeplitz (or Hankel) matrix in our preferred pattern to
perform the similarity transformation. For example, if we want to find a lower
triangular Toeplitz matrix $T$ such as $T^{-1} C_{t}(\mathbf{u}) T =
C_{r}(\mathbf{u})$, it turns out that $T$ is actually a scalar multiple of
$U_{+}$. This also confirms that (\ref{similar1}) is a special case of
Theorem \ref{similarbyt}. We demonstrate this in the following example instead of
giving a general proof. The general proof follows from the same argument.

We will now illustrate this with an example.  Let $\mathbf{u} = [4,3,2,1]^{T}$. Then
\[
C_{t}(\mathbf{u}):=\left[
\begin{array}
[c]{rrr}%
-2 & -3 & -4\\
1 & 0 & 0\\
0 & 1 & 0
\end{array}
\right]  \quad\mbox{and} \quad C_{r}(\mathbf{u}):=\left[
\begin{array}
[c]{rrr}%
0 & 0 & -4\\
1 & 0 & -3\\
0 & 1 & -2
\end{array}
\right]  .
\]
Suppose we wish to construct a lower triangular Toeplitz $T$ such that
\[
T^{-1} C_{t}(\mathbf{u}) T = C_{r}(\mathbf{u}) .
\]
First we find a Toeplitz $T_{1}$ by using the procedure given in the proof.
All we need to do is to find a $\partial T_{1}$ and then obtain $T_{1}$ from
$\partial T_{1}$. $\partial T_{1}$ is a $2\times4$ Toeplitz matrix whose rows
are both perpendicular to $\mathbf{u}$. For convenience, we choose the first
row to be $\left[
\begin{array}
[c]{rrrr}%
1/4 & 0 & 0 & -1
\end{array}
\right]  $. Then $\partial T_{1}$ takes the form
\[
\left[
\begin{array}
[c]{rrrr}%
1/4 & 0 & 0 & -1\\
x & 1/4 & 0 & 0
\end{array}
\right]
\]
where $x$ is determined so that the second row is also perpendicular to
$\mathbf{u}$, and hence $x =-3/16$. Therefore
\[
T_{1} = \left[
\begin{array}
[c]{rrr}%
0 & 0 & -1\\
1/4 & 0 & 0\\
-3/16 & 1/4 & 0
\end{array}
\right]  .
\]
To obtain a lower triangular Toeplitz matrix we apply Corollary \ref{shift} to
$T_{1}$:
\[
T = T_{1} C_{l} = \left[
\begin{array}
[c]{rrr}%
1/4 & 0 & 0\\
-3/16 & 1/4 & 0\\
1/64 & -3/16 & 1/4
\end{array}
\right]  .
\]
We can check that $T=U_{+}^{-1}$ and hence the similarity (\ref{similar1}) holds.

Now we turn to the special case (\ref{similar2}) of Theorem \ref{similarbyt},
where the similarity is carried out by the Toeplitz Bezoutian. For this
purpose we first collect some results from Corollary 2.3, 2.10, Theorem 4.2
and 4.5 of \cite{Bezou:1} and summarize them in the following Theorems.

\begin{theorem}
\label{quoteThm} A necessary and sufficient condition for two non-zero
Toeplitz Bezoutian matrices $B_{T}(\mathbf{a},\mathbf{b})$ and $B_{T}(\mathbf{a}%
_{1},\mathbf{b}_{1})$  to
coincide is
\[
\left[
\begin{array}
[c]{cc}%
\mathbf{a}_{1} & \mathbf{b}_{1}%
\end{array}
\right]  =\left[
\begin{array}
[c]{cc}%
\mathbf{a} & \mathbf{b}%
\end{array}
\right]  \varphi
\]
for some matrix $\varphi$ with $\det\varphi=1$.

If $T$ is an invertible Toeplitz matrix and
$\{\mathbf{a}, \mathbf{b}\}$ is a basis for the kernel of $\partial T$. Then $B_{T}(\mathbf{a},\mathbf{b})$
is just a scalar multiple of $T^{-1}$.
\end{theorem}

\begin{theorem}
\label{quoteThm2} A necessary and sufficient condition for two non-zero
Hankel Bezoutian matrices $B_{H} (\mathbf{a},\mathbf{b})$ and $B_{H}(\mathbf{a}_{1},\mathbf{b}_{1})$ to
coincide is
\[
\left[
\begin{array}
[c]{cc}%
\mathbf{a}_{1} & \mathbf{b}_{1}%
\end{array}
\right]  =\left[
\begin{array}
[c]{cc}%
\mathbf{a} & \mathbf{b}%
\end{array}
\right]  \varphi
\]
for some matrix $\varphi$ with $\det\varphi=1$.

If $H$ is an invertible  Hankel matrix and
$\{\mathbf{a}, \mathbf{b}\}$ is a basis for the kernel of
$\partial H$. Then $B_{H}(\mathbf{a},\mathbf{b})$
is just a scalar multiple of  $H^{-1}$.
\end{theorem}

Putting $T=B_{T}^{-1}$ in Theorem \ref{quoteThm} we see that if $\{\mathbf{a}, \mathbf{b}\}$ is a basis for the kernel of $\partial(B_{T}^{-1})$ then $B_{T}(\mathbf{a},\mathbf{b})=\lambda B_{T}(\mathbf{u},\mathbf{v})$ for some nonzero constant $\lambda$.  By Theorem \ref{quoteThm} again we have
\[
\left[
\begin{array}
[c]{cc}%
\mathbf{u} & \mathbf{v}%
\end{array}
\right]  = \left[
\begin{array}
[c]{cc}%
\mathbf{a} & \mathbf{b}%
\end{array}
\right]  \varphi
\]
for some invertible matrix $\varphi$. This means that both $\mathbf{u}$ and $\mathbf{v}$ are in the kernel of $\partial T$ and hence Theorem \ref{similarbyt}
implies
 \begin{equation}
\label{similargenBt}B_{T} C_{t}(\mathbf{u})^{k} B_{T}^{-1} = C_{r}%
(\mathbf{u})^{k}, \quad B_{T} C_{t}(\mathbf{v})^{k} B_{T}^{-1} =
C_{r}(\mathbf{v})^{k}%
\end{equation}
for all integers $k$.
Now we can show that relation (\ref{similar2}) is nothing but the first
equation in (\ref{similargenBt}). To see this we rewrite the first equation in the
form $C_{t} = B_{T}^{-1} C_{r} B_{T}$ and then take transpose
\[
\overline{C}_{l}=C_{t}^{T}=C_{r}^{J}=B_{T}^{T} C_{r}^{T} (B_{T}^{-1})^{T}
=(B_{T}^{T}J)(C_{r})^{J} (B_{T}^{T} J)^{-1}=QC_{t}Q^{-1}%
\]
which is (\ref{similar2}).

\subsection{Extension using companion matrices}

A Toeplitz (or Hankel) matrix can be extended to any size in a Toeplitz (or
Hankel) way, by adding more diagonal bands to the existing bands. Theoretical
aspects of such an extension, such as the minimum rank of the extension, have
been studied in the literature (see \cite{Alpin:1} and the references therein).
What we are concerned with here is a specific way of extending the matrix by
multiplication by some associated companion matrices. We hope this extension
might have more applications than the ones we will demonstrate at the end of
this paper. We will use the similarity relations among companion matrices that
have been developed earlier.

Assume $A$ is a $n\times n$ matrix. The role that $C_{t}$ plays in the product
$C_{t}A$ is to keep the first $n-1$ rows of $A$ as the last $n-1$ rows of
$C_{t}A,$ and to add one new row on the top. The new row added is a linear
combination of rows of $A$. Similarly, the first $n-1$ rows of $C_{b}A$ are
the last $n-1$ rows of $A$ and the last row of $C_{b}A$ is a linear
combination of rows of $A$. This enables us to extend the matrix $A$ in the
upward and downward directions as follows. Starting from $A$, for integers $k\geq l$ we define $\mathcal{T}[A:k,l]$ to be the $(n+k-l) \times n$ matrix
\begin{equation}
\mathcal{T}[A:k,l] = \left[
\begin{array}
[c]{c}%
\gamma_{k-l}\\
\vdots\\
\gamma_{1}\\
C_{t}^{l}A
\end{array}
\right]   \label{vexten}%
\end{equation}
where $\gamma_{i}$ ($i=1,2,\ldots,k-l$) is the first row of $C_{t}^{l+i}A$.

In similar fashion the effect of post-multiplying a matrix by $C_{r}$ or
$C_{l}$ can be considered. We can extend a matrix in the right and left
directions by adding the last column of $AC_{r}$ to the right or adding the
first column of $AC_{l}$ to the left. Starting from $\mathcal{T}[A:k,l]$, for
integers $s \geq t$ we define
\begin{equation}
\mathcal{T}[A:k,l;s,t]=\left[
\begin{array}
[c]{cccc}%
\mathcal{T}[A:k,l]C_{r}^{t} & \beta_{1} & \cdots &
\beta_{s-t}
\end{array}
\right]  ,
\label{hexten}%
\end{equation}
where $\beta_{i}$ ($i=1,2,\ldots,s-t$) is the last column of $\mathcal{T}%
[A:k,l]C_{r}^{t+i}$. We call this a Toeplitz extension
because we will prove that, under certain conditions, such an extension
preserves the Toeplitz structure if the starting matrix $A$ is Toeplitz. We
call $A$ a generator in such an extension. To generate the same matrix
$\mathcal{T}[A:k,l;s,t]$ we can use any $n\times n$ matrix of the form
$C_{t}^{i}AC_{r}^{j}$ where $i$ and $j$ are integers. It is easy to see that
\[
\mathcal{T}[A:k,l;s,t]=\mathcal{T}[C_{t}^{i}AC_{r}^{j}:k-i,l-i;s-j,t-j].
\]

If we use $\overline{C_{r}}$ and $\overline{C_{l}}$ instead of $C_{r}$ and
$C_{l}$ in the above extension, we will obtain a different extended matrix
$\mathcal{H}[A:k,l;s,t]$. We call this a Hankel extension because it preserves
the Hankel structure under certain conditions.

Here are two examples:
\[
\mathcal{T}[I:n,-n;n,-n]=\left[
\begin{array}
[c]{ccc}%
C_{t}^{n}C_{l}^{n} & C_{t}^{n} & C_{t}^{n}C_{r}^{n}\\
C_{l}^{n} & I & C_{r}^{n}\\
C_{b}^{n}C_{l}^{n} & C_{b}^{n} & C_{b}^{n}C_{r}^{n}%
\end{array}
\right]  ,\]
\[
\mathcal{H}[A:n,-n;-n,-2n]=\left[
\begin{array}
[c]{cc}%
 C_{t}^{n}A\overline{C}_{r}^{-2n} & C_{t}^{n}A\overline{C}_{r}^{-n} \\
 A\overline{C}_{r}^{-2n} & A\overline{C}_{r}^{-n}\\
 C_{b}^{n}A\overline{C}_{r}^{-2n} & C_{b}^{n}A\overline{C}_{r}^{-n}
\end{array}
\right] =\left[
\begin{array}
[c]{cc}%
 C_{t}^{n}A\overline{C}_{l}^{2n} & C_{t}^{n}A\overline{C}_{l}^{n} \\
 A\overline{C}_{l}^{2n} &  A\overline{C}_{l}^{n} \\
 C_{b}^{n}A\overline{C}_{l}^{2n} & C_{b}^{n}A\overline{C}_{l}^{n}
\end{array}
\right]  .
\]
We notice that, for any square matrix $A$,
\[
\mathcal{T}[A:k,l;s,t]=\mathcal{T}[I:k,l;0,0]A\mathcal{T}[I:0,0;s,t].
\]
An obvious property of these extensions is given in the following Proposition.

\begin{proposition}
\label{ker} Suppose $A$ is invertible. Then the rank of $\mathcal{T}%
[A:k,l; s, t]$ is $n$. If $r=s-t>0$ then $\{
e_{1}, \ldots, e_{r}\}$ is a basis for the kernel of $\mathcal{T}[A:k,l; s,
t]$, where $e_{i}$ is the $i$th column of the
$(n+r)\times r$ Toeplitz matrix whose first column is $\left[
\begin{array}
[c]{cccccc}%
u_{1} & \cdots & u_{n+1} & 0 & \cdots & 0
\end{array}
\right]  ^{T} $ and last column is $\left[
\begin{array}
[c]{cccccc}%
0 & \cdots & 0 & u_{1} & \cdots & u_{n+1}%
\end{array}
\right]  ^{T} $.  In particular, $\mathcal{T}[A:k,k; s,s-1] \vu =0$ for all integers $k$ and $s$.
\end{proposition}

\begin{proof}
All the rows of $\mathcal{T}[A:k,l; s, t]$ are linear combinations of the rows of $\mathcal{T}[A:0,0; s, t]$ which is a rank $n$ matrix. Thus  $\mathcal{T}[A:k,l; s, t]$ is of rank $n$. It also follows that the kernel of  $\mathcal{T}[A:k,l; s, t]$ is the same as the kernel of  $\mathcal{T}[A:0,0; s, t]$.  The latter is an $n\times (n+r)$ matrix so its kernel has dimension $r$.  It is clear that the set $\{ e_1, \ldots, e_{r}\}$ is linearly independent.  So the only thing we need to verify is $
\mathcal{T}[A:0,0; s, t]e_i =0$ for $i=1, \ldots, r$. Due to the structure of $e_i$
	\[ \mathcal{T}[A:0,0; s, t]e_i =  \left[
		\begin{array}
		[c]{cc} A C_r^{s+i-1} & \va
		\end{array}
		\right] \vu \]
where $\va$ is the last column of $A C_r^{s+i} = A C_r^{s+i-1} C_r$.  Therefore $\va = A C_r^{s+i-1} \vb $ where $\vb$ is the last column of $C_r$.  A direct verification yields $ \left[	 \begin{array}[c]{cc} I & \vb \end{array} \right] \vu =0$.  As a consequence
	\[   \left[	\begin{array}
			[c]{cc} A C_r^{s+i-1} & \va
			\end{array} \right] \vu
	= \left[ \begin{array}
				[c]{cc} A C_r^{s+i-1} & A C_r^{s+i-1} \vb
			\end{array}  \right] \vu
	= A C_r^{s+i-1} \left[ \begin{array}
					[c]{cc}I &  \vb
				\end{array}  \right] \vu =0.	
	 \]				
\end{proof}

\begin{corollary}
\label{hker} Suppose $A$ is invertible. Then the rank of $\mathcal{H}[A:k,l; s, t] $ is $n$. If $r=s-t>0$ then $\{ e_{1}^{J}, \ldots, e_{r}^{J}\}$ is a basis for the kernel of
$\mathcal{H}[A:k,l; s, t]$, where $e_{i}^{J}$ is the $i$th column of the
$(n+r)\times r$ Hankel matrix whose first column is $\left[
\begin{array}
[c]{cccccc}%
0 & \cdots & 0 & u_{n+1} & \cdots & u_{1}%
\end{array}
\right]  ^{T} $ and last column is $\left[
\begin{array}
[c]{cccccc}%
u_{n+1} & \cdots & u_{1} & 0 & \cdots & 0
\end{array}
\right]  ^{T} $.
\end{corollary}
\begin{proof}
Use $\overline{C}_r$ and $e_i^J$ instead of $C_r$ and $e_i$ in the proof of Proposition \ref{ker}.
\end{proof}

The following Lemma is a preparation for the proof of our main Theorem \ref{extension}.

\begin{lemma}
\label{tkershift} Let $T$ be a Toeplitz matrix and $\mathbf{u} =(u_{1}%
,\ldots,u_{n+1})^{T}$ be a vector such that $u_{1}, u_{n+1} \neq0$. If
$\mathbf{u}$ belongs to the kernel of $\partial T$ then $\mathbf{u}$ also
belongs to the kernels of $\partial(C_{t} T)$, $\partial(C_{b} T)$, $\partial(T C_{r} T)$ and $\partial(T C_{l} T)$.
\end{lemma}

\begin{proof}
We only prove the case $\partial(C_{t} T)$; the proof for the case $\partial(C_{b} T)$ is similar.  The other two cases are covered by Theorem \ref{similarbyt}. Let
\[  T = \left[
\begin{array}{ccc}
a_0&  \cdots  & a_{1-n}  \\
\vdots  &\ddots  & \vdots \\
a_{n-1} &  \cdots &  a_0
\end{array}
\right] \]
then, By Theorem \ref{similarbyt}, $C_{t} T$ is Toeplitz and hence
\[ C_t T =  \left[ \begin{array}{c}
\begin{array}{cccc} a_{-1} &  \cdots  & a_{1-n} & \mu_{-1} \end{array} \\
T_{[1,n-1,1,n]}
\end{array}
\right], \]
where
\[ \begin{array}{c} \dst \mu_{-1} = -\frac{1}{u_{n+1}} [u_n, \cdots , u_1] [ a_{1-n},\cdots ,a_0]^T.
\end{array}\]
It follows that
\[ \partial (C_t T) = \left[
\begin{array}{c}
\begin{array}{cccc}a_0 &  \cdots  & a_{1-n} & \mu_{-1} \end{array} \\
S_1
\end{array}
\right], \]
where $S_1$ is the sub-matrix of $\partial T$ consisting of the first $n-2$ rows of $\partial T$.
From the definition of $\mu_{-1}$ we have
\[ \left[ \begin{array}{ccccc}a_0 & a_{-1}&  \cdots  & a_{1-n} & \mu_{-1} \end{array} \right] \vu = 0 \]
and hence
\[  \partial (C_t T)\vu = 0.\]
\end{proof}

Putting $H=TJ$, we have immediately
\begin{corollary}
\label{hkershift} Let $H$ be a Hankel matrix and $\mathbf{u} =(u_{1}%
,\ldots,u_{n+1})^{T}$ be a vector such that $u_{1}, u_{n+1} \neq0$. If
$\mathbf{u}^J$ belongs to the kernel of $\partial H$ then $\mathbf{u}^{J}$ belongs to the kernels of $\partial(C_{t}H)$
and $\partial(C_{b}H)$.
\end{corollary}

The more interesting features of the extensions are now presented.

\begin{theorem}
\label{extension} Let $T$ be an invertible Toeplitz matrix and $\mathbf{u}
=(u_{1},\ldots,u_{n+1})^{T}$ be a vector such that $u_{1}, u_{n+1} \neq0$. If
$\mathbf{u}$ belongs to the kernel of $\partial T$, then the matrix
$\mathcal{T}[T:k,l; s, t]$ is Toeplitz.

\end{theorem}

\begin{proof}  Because any $n \times n$ block in $\mathcal{T}[T:k,l; s, t]$ is in the form of $C_t^iTC_r^j$ for some integers $i$ and $j$, we only need to prove that all such blocks are Toeplitz.  We use the same argument in all the four directions of extension and only demonstrate this argument in one direction, say, the direction to the right.  Without loss of generality we assume $i=0$ and we prove that $TC_r^j$ is Toeplitz by induction on $j>0$. Theorem \ref{similarbyt} has already covered the case $j=1$.  Assume now all matrices $TC_r^s$, $s=0, 1, \ldots, j$ are Toeplitz.  Then Lemma \ref{tkershift} guarantees that $\mathbf{u}$ belongs to the kernel of $\partial (TC_r^j)$. Finally by Theorem \ref{similarbyt} we conclude that $\partial (TC_r^{j+1})$ is Toeplitz.

In the direction to the left when $j<0$, the same induction argument proves the case $j-1$.  The same argument works in the direction of up and down extensions.

\end{proof}

\begin{corollary}
Let $H$ be an invertible Hankel matrix and $\mathbf{u}
=(u_{1},\ldots,u_{n+1})^{T}$ be a vector such that $u_{1}, u_{n+1} \neq0$. If
$\mathbf{u}^J$ belongs to the kernel of $\partial H$, then the matrix
$\mathcal{H}[H:k,l;s,t]$ is Hankel.
\end{corollary}
\begin{proof}
Define $T=HJ$ then $T$ satisfies the conditions of Theorem \ref{extension} and hence this Corollary follows.
\end{proof}

\section{Examples and applications}

\subsection{Examples}

For an invertible Toeplitz matrix $T$ the kernel of $\partial T$ is
$2$-dimensional and there are infinitely many choices of bases $\{\mathbf{u}%
,\mathbf{v}\}$ for $\partial T$. For any such $\mathbf{u}$ or $\mathbf{v}$ a
companion matrix can be constructed which can be used to build extensions.
Here we give three examples of obvious extensions.

\textbf{Example 1.} $\mathcal{T}[I:k,l; s, t]$. This extension is not
necessarily Toeplitz even if the starting matrix $I$ is Toeplitz, because
$\mathbf{u}$ is not in the kernel of $\partial I$ unless $u_{2}=\cdots=
u_{n}=0$. However this matrix has direct application in state evolution of a
dynamic system, as given later.

\textbf{Example 2.} $\mathcal{T}[U_{+}^{-1}:k,l;s,t]$. $U_{+}^{-1}$ is lower
triangular and we denote its elements in the first column by $s_{1}$
$\ldots$, $s_{n}$. Then $\partial U_{+}^{-1}$ has the form
\[
\partial U_{+}^{-1}:=\left[
\begin{array}
[c]{ccccc}%
s_{2} & s_{1} & 0 & \cdots & 0\\
\vdots & \ddots & \ddots & \ddots & \vdots\\
s_{n} & \cdots & s_{2} & s_{1} & 0
\end{array}
\right]
\]
and, obviously, $\mathbf{u}$ is in the kernel of $\partial U_{+}^{-1}$.
Therefore, $\mathcal{T}[U_{+}^{-1}:k,l;s,t]$ is Toeplitz. If $m=\max\{k,s\}>0$
and $h=\max\{l,t\}>0$, it can be shown that such an extension has the
following features: (a) there is a middle band consisting of $n-1$ diagonal
lines of zeros; (b) the elements below the middle band of zeros are the elements (displayed in the same order) of the inverse of the $(n+h)\times (n+h)$ lower triangular Toeplitz matrix whose first column is $[u_{1},\ldots,u_{n+1}, 0, \ldots, 0]^T$; (c) the elements above the middle band of zeros are the elements (displayed in the same order) of the inverse of the $m\times m$ upper triangular Toeplitz matrix whose first row is the truncation of the first $m$ elements of $[-u_{n+1},\ldots,-u_{1}, 0,\ldots ]$. We skip the proof of this feature but demonstrate it in the case of $k=s=n$, $t=-n$ and $l=0$. The general proof can be written down using a similar argument. We write
\[
\mathcal{T}[U_{+}^{-1}:n,0;n,-n]=\left[
\begin{array}
[c]{ccc}%
S_{1} & R_{1} & R_{2}\\
S_{2} & S_{1} & R_{1}%
\end{array}
\right]  =\left[
\begin{array}
[c]{ccc}%
S_{1} & 0 & 0\\
S_{2} & S_{1} & 0
\end{array}
\right]  +\left[
\begin{array}
[c]{ccc}%
0 & R_{1} & R_{2}\\
0 & 0 & R_{1}%
\end{array}
\right]  .
\]
It is easy to see that
\[
\left[
\begin{array}
[c]{cc}%
S_{1} & 0\\
S_{2} & S_{1}%
\end{array}
\right]  =\left[
\begin{array}
[c]{cc}%
U_{+} & 0\\
U_{-} & U_{+}%
\end{array}
\right]  ^{-1}%
\]
because $S_{1}=U_{+}^{-1}$ and, by Proposition \ref{ker}, $S_{2}U_{+}%
+S_{1}U_{-}=0$. We now show that
\[
\left[
\begin{array}
[c]{cc}%
R_{1} & R_{2}\\
0 & R_{1}%
\end{array}
\right]  =-\left[
\begin{array}
[c]{cc}%
U_{-} & U_{+}\\
0 & U_{-}%
\end{array}
\right]  ^{-1}.
\]
To see this we apply Proposition \ref{ker} to $\mathcal{T}[U_{+}%
^{-1}:n,0;n,-n]$:
\begin{align*}
0 &  =\left(  \left[
\begin{array}
[c]{ccc}%
S_{1} & 0 & 0\\
S_{2} & S_{1} & 0
\end{array}
\right]  +\left[
\begin{array}
[c]{ccc}%
0 & R_{1} & R_{2}\\
0 & 0 & R_{1}%
\end{array}
\right]  \right)  \left[
\begin{array}
[c]{cc}%
U_{+} & 0\\
U_{-} & U_{+}\\
0 & U_{-}%
\end{array}
\right]  \\
&  =\left[
\begin{array}
[c]{cc}%
S_{1} & 0\\
S_{2} & S_{1}%
\end{array}
\right]  \left[
\begin{array}
[c]{cc}%
U_{+} & 0\\
U_{-} & U_{+}%
\end{array}
\right]  +\left[
\begin{array}
[c]{cc}%
R_{1} & R_{2}\\
0 & R_{1}%
\end{array}
\right]  \left[
\begin{array}
[c]{cc}%
U_{-} & U_{+}\\
0 & U_{-}%
\end{array}
\right]  .
\end{align*}
The first term on the right hand side is equal to $I$ and hence the second
term on the right hand side is equal to $-I$.

\textbf{Example 3.} $\mathcal{T}[B_{T}^{-1}:k,l;s,t]$. Since both $\mathbf{u}$
and $\mathbf{v}$ are in the kernel of $\partial B_{T}^{-1}$ there are two
different version of this extension: $\mathcal{T}_{u}[B_{T}^{-1}:k,l;s,t]$ by
using $C_{t}(\mathbf{u})$ and $C_{r}(\mathbf{u})$, and, $\mathcal{T}_{v}%
[B_{T}^{-1}:k,l;s,t]$ by using $C_{t}(\mathbf{v})$ and $C_{r}(\mathbf{v})$.
Both $\mathcal{T}_{u}[B_{T}^{-1}:k,l;s,t]$ and $\mathcal{T}_{v}[B_{T}%
^{-1}:k,l;s,t]$ are Toeplitz and share the same central band (the diagonal
band that contains only all the diagonal lines of $B_{T}^{-1}$).

\subsection{Applications}

For an arbitrary sequence
$y=(y_{k})_{k=1}^{\infty}$ its $\lambda$-transform (generating function) is defined to be $\hat{y}(\lambda):=\sum _{k=1}^{\infty}y_{k}\lambda^{k-1}$.  Consider a linear, time-invariant, causal, discrete-time dynamic system with
transfer function description $\hat{y}(\lambda)=-(v%
(\lambda)/u(\lambda))\hat{x}(\lambda)$, where $x=(x_{k}%
)_{k=1}^{\infty}$ is the input, $y=(y_{k})_{k=1}^{\infty}$ is the output, and
the numerator and the denominator of the transfer function $-v%
(\lambda)/u(\lambda)$ satisfy all the assumptions stated in Section
1.   We will represent this system in state space form by introducing state
vectors first and then write down the rule of evolution of state vectors in
terms of the initial state and the input. Known properties of the transition
matrix will be derived as a special case of our results on extension of
Toeplitz matrices.

When both $x$ and $y$ are sequences in $ l_{1}$, the analytic functions $\hat{x}(\lambda)$ and $\hat{y}(\lambda)$ are
related by
\[
u(\lambda)\hat{y}(\lambda)=-v(\lambda)\hat{x}(\lambda).
\]
Equating like powers of $\lambda$ gives
\[
Uy+Vx=0
\]
where $x$ and $y$ are  columns and
\[
U:=\left[
\begin{array}
[c]{cccc}%
U_{+} &  &  & 0\\
U_{-} & U_{+} &  & \vspace{-2mm}\\
& U_{-} & \ddots & \\
0 &  & \ddots & \ddots
\end{array}
\right]  \quad\quad V:=\left[
\begin{array}
[c]{cccc}%
V_{+} &  &  & 0\\
V_{-} & V_{+} &  & \vspace{-2mm}\\
& V_{-} & \ddots & \\
0 &  & \ddots & \ddots
\end{array}
\right]  .
\]
It can be shown using functional analysis arguments that such $x$ and $y$ have the form
\begin{equation}
(x,y)=(-Ub,Vb)\quad\mbox{for some}\quad
b\in l_{1}. \label{xyeqn}%
\end{equation}
Writing out this equation in detail yields the difference equations
\begin{equation}
x_{k}=-u_{n+1}b_{k}-u_{n}b_{k+1}-\cdots-u_{2}b_{k+n-1}-u_{1}b_{k+n}
\label{xinput}%
\end{equation}
and
\begin{equation}
y_{k}=v_{n+1}b_{k}+v_{n}b_{k+1}+\cdots+v_{2}b_{k+n-1}+v_{1}b_{k+n}.
\label{youtput}%
\end{equation}
Now we can introduce naturally the $n$-dimensional state vector at the time $k$ as the
truncation $b_{[k;n+k-1]}$ of $b$ and denote it by
\[
\mathbf{\beta}(k)=[b_{k},b_{k+1},\ldots,b_{k+n-1}]^{T}.
\]
Then we can put (\ref{xinput}) and (\ref{youtput}) in the state space form
\begin{equation} \label{state}
\mathbf{\beta}(k+1)=C_{b}\mathbf{\beta}(k)+Bx_{k},\quad y_{k}=D
\mathbf{\beta}(k)-\frac{v_{1}}{u_{1}}x_{k}
\end{equation}
where the input matrix $B$ and the output matrix $D$ are given by
\[
B=-\left[
\begin{array}
[c]{c}%
0\\
\vdots\\
0\\
1/u_{1}%
\end{array}
\right]  \quad\mbox{and}\quad D=\frac{1}{u_{1}}\left[
\begin{array}
[c]{ccc}%
u_{1}v_{n+1}-v_{1}u_{n+1} & \cdots & u_{1}v_{2}-v_{1}u_{2}%
\end{array}
\right]  .
\]
Note that the output matrix $D$ is actually the first row of $B_{T}%
(\mathbf{u},\mathbf{v})$ divided by $u_{1}$. Using the extension given in
Example 1, the general truncation $b_{[1:n+p]}=[b_{1},\ldots,b_{n+p}%
]^{T}$ for $p>0$ is then given by
\begin{equation}
b_{[1:n+p]}=\mathcal{T}[I:0,-p;0,0]\mathbf{\beta}(0)-\frac{1}{u_{1}}\left[
\begin{array}
[c]{c}%
O_{n\times p}\\
F_{p}%
\end{array}
\right]  \left[
\begin{array}
[c]{c}%
x_{1}\\
\vdots\\
x_{p}%
\end{array}
\right]  , \label{longstate}%
\end{equation}
where
\[
F_{p}=\left[
\begin{array}
[c]{cccc}%
1 &  &  & 0\\
s_{1} & 1 &  & \\
\vdots & \ddots & \ddots & \\
s_{p-1} & \cdots & s_{1} & 1
\end{array}
\right]  ,
\]
and $s_{i}$ is the element at the last column and last row of $(C_{b})^{i}$.
It can be shown that $(1/u_{1})F_{p}$ is the inverse of the $p\times p$ lower
triangular nonsingular truncation of $U$, but we skip the proof here.

Now we change the basis of the state space by using the transforming matrix
$B_{T}$, that is, we introduce $\mathbf{\beta}^{\prime}(k)=B_{T}\mathbf{\beta}(k)$.
Then, by (\ref{similargenBt}), the state space representation of the
system is transformed into another canonical form
\begin{equation}
\mathbf{\beta}^{\prime}(k+1)=C_{l}\mathbf{\beta}^{\prime}(k)+B_{1}x_{k},\quad
y_{k}=D_{1}\mathbf{\beta}^{\prime}(k)-\frac{v_{1}}{u_{1}}x_{k} \label{state1}%
\end{equation}
where the input matrix $B_{1}$ and the output matrix $D_{1}$ now are given by
\[
B_{1}=-\frac{1}{u_{1}}\mbox{Last column of $B_T$}\quad\mbox{and}\quad
D_{1}=\frac{1}{u_{1}}\left[
\begin{array}
[c]{cccc}%
1 & 0 & \cdots & 0
\end{array}
\right]  .
\]
Then the state vector at the time $q>0$ can be expressed directly in terms of
the initial state $\mathbf{\beta}^{\prime}(0)$ and the input data:
\begin{equation}
\mathbf{\beta}^{\prime}(q)=C_{l}^{q}\mathbf{\beta}^{\prime}(0)+\mathcal{T}%
[I:0,0;0,1-q]E_{q}\left[
\begin{array}
[c]{c}%
x_{1}\\
\vdots\\
x_{q}%
\end{array}
\right]  , \label{latestate}%
\end{equation}
where $E_{q}$ is the $(n+q-1)\times q$ band matrix
\[
\left[
\begin{array}
[c]{ccc}%
B_{1} &  & 0\\
& \ddots & \\
0 &  & B_{1}
\end{array}
\right]  .
\]
Combining (\ref{longstate}) and ({\ref{latestate}) gives the mixed case,
evolving the state vectors under the basis above until a given time $q$, then changing
the basis and evolving further to the time $q+p$.
\begin{align}
b_{[q+1:n+q+p]}  &  =\mathcal{T}[I:0,-p;0,0]B_{T}^{-1}\mathbf{\beta}
^{\prime}(q)-\frac{1}{u_{1}}\left[
\begin{array}
[c]{c}%
O_{n\times p}\\
F_{p}%
\end{array}
\right]  \left[
\begin{array}
[c]{c}%
x_{q+1}\\
\vdots\\
x_{q+p}%
\end{array}
\right] \nonumber\label{longlatestate}\\
&  =\mathcal{T}[B_{T}^{-1}:0,-p;n-q,-q]\mathbf{\beta}^{\prime}(0)+\mathcal{T}%
[I:0,-p;0,1-q]E_{q}\left[
\begin{array}
[c]{c}%
x_{1}\\
\vdots\\
x_{q}%
\end{array}
\right] \nonumber\\
&  -\frac{1}{u_{1}}\left[
\begin{array}
[c]{c}%
O_{n\times p}\\
F_{p}%
\end{array}
\right]  \left[
\begin{array}
[c]{c}%
x_{q+1}\\
\vdots\\
x_{q+p}%
\end{array}
\right]
\end{align}
for all positive integers $p$ and $q$. The expression in (\ref{longlatestate})
is for demonstration only. It is not necessary in practice because there is no
need to change basis in the middle of evolution. }

All extensions involve powers of companion matrices. The formula derived in
\cite{Lim:1} can be used to calculate entries of an integer power of a
companion matrix directly and hence it is possible to calculate all required
entries of our extension directly without calculating all the powers of the
companion matrix.


\end{document}